\documentclass[11pt]{amsart}

\usepackage{latexsym,rawfonts}
\usepackage{amsfonts,amssymb}
\usepackage{amsmath,amsthm}

\usepackage[plainpages=false]{hyperref}
\usepackage{graphicx}
\usepackage{color}
\usepackage[table]{xcolor}
\usepackage{longtable}

\newtheorem{theorem}{Theorem}[section]

\newtheorem{lemma}{Lemma}[section]

\theoremstyle{definition}

\numberwithin{equation}{section}

\newcommand{\R}{\mathbb{R}}

\newcommand{\dd}{\mathop{}\!\mathrm{d}}

\newcommand{\set}[1]{\left\{#1\right\}}

\newcommand{\pd}{\partial}

\newcommand{\uS}{\mathbb{S}^{n-1}}

\newcommand{\MA}{Monge-Amp\`ere }

\newcommand{\Sy}{\mathbb{S}^1}
\newcommand{\II}{I\!\!I}
\newcommand{\III}{I\!\!I\!\!I}
\newcommand{\pdR}{\pd_{\!_R}}
\newcommand{\UB}{\overline{\Lambda_k}}
\newcommand{\LB}{\underline{\lambda_k}}

\begin{document}

\title{On the number of solutions to the planar dual Minkowski problem}

\author{YanNan Liu}
\address{School of Mathematics and Statistics, Beijing Technology and Business University, Beijing 100048, P.R. China}
\email{liuyn@th.btbu.edu.cn}

\author{Jian Lu}
\address{School of Mathematical Sciences, South China Normal University, Guangzhou 510631, P.R. China}
\email{lj-tshu04@163.com}

\thanks{The first author was supported by
  Natural Science Foundation of China (12071017 and 12141103), and
  Natural Science Foundation of Beijing Municipality (1212002).
  The second author was supported by
  Natural Science Foundation of China (12122106).}

\date{}

\begin{abstract}
  The dual Minkowski problem in the two-dimensional plane is studied in this
  paper. By combining the theoretical analysis and numerical estimation of an
  integral with parameters, we find the number of solutions to this problem for
  the constant dual curvature case when $0<q\leq4$. An improved nonuniqueness
  result when $q>4$ is also obtained. As an application, a result on the
  uniqueness and nonuniqueness of solutions to the $L_p$-Alexandrov problem is
  obtained for $p<0$. 
\end{abstract}


\makeatletter
\@namedef{subjclassname@2020}{\textup{2020} Mathematics Subject Classification}
\makeatother


\maketitle
\vskip4ex

\section{Introduction}

The dual Minkowski problem was first proposed by Huang-Lutwak-Yang-Zhang in
their groundbreaking paper \cite{HLYZ.Acta.216-2016.325} in 2016.
It asks what are the necessary and sufficient conditions on a Borel measure $\mu$
defined on the unit sphere $\uS$ such that $\mu$ is the dual curvature
measure of some convex body $K$ in the Euclidean space $\R^n$.
When the given measure $\mu$ is absolutely continuous with respect to the
standard measure on $\uS$, the \emph{dual Minkowski problem} is equivalent to solving
the following \MA type equation: 
\begin{equation} \label{dMP}
  h (|\nabla h|^2+h^2)^{\frac{q-n}{2}} \det(\nabla^2h+hI) = f \quad \text{ on } \ \uS,
\end{equation}
where $h$ is the support function of some unknown convex body $K$ in $\R^n$,
$\nabla$ is the covariant derivative with respect to an orthonormal frame on
$\uS$, $q\in\R$ is a given number, $I$ is the unit matrix of order $n-1$, and
$f$ is a given nonnegative integrable function on $\uS$. 
Note that $\nabla h(x) +h(x) x$ is the point on $\pd K$ whose unit outer normal
vector is $x\in\uS$.

The dual Minkowski problem unifies two important prescribed curvature problems.
One is the logarithmic Minkowski problem, corresponding to the case when $q=n$;
see e.g. 
\cite{BHZ.IMRNI.2016.1807,
  BLYZ.JAMS.26-2013.831,
  CLZ.TAMS.371-2019.2623,
  KM.MAMS.277-2022.1,
  Sta.Adv.180-2003.290,
  Zhu.Adv.262-2014.909}.
The other is the Alexandrov problem \cite{Ale.CRDASUN.35-1942.131}, which is the
prescribed Alexandrov integral curvature problem and corresponds to the case
when $q=0$. 
Huang-Lutwak-Yang-Zhang \cite{HLYZ.Acta.216-2016.325} proved a variational
formula for the dual Minkowski problem, based on which a result on the existence
of even solutions was obtained. 
Since then, there have been many studies on this problem
\cite{BHP.JDG.109-2018.411,
  BLY+.Adv.356-2019.106805,
  CL.Adv.333-2018.87,
  HP.Adv.323-2018.114,
  HJ.JFA.277-2019.2209,
  JW.JDE.263-2017.3230,
  LSW.JEMSJ.22-2020.893,
  Zha.CVPDE.56-2017.18,
  Zha.JDG.110-2018.543}.
In addition, various extensions of the dual Minkowski problem have also received
a lot of attention; see e.g.
\cite{BF.JDE.266-2019.7980,
  CHZ.MA.373-2019.953,
  CCL.AP.14-2021.689,
  CL.JFA.281-2021.109139,
  CLLX.JGA.32-2022.40,
  GHW+.CVPDE.58-2019.12,
  GHXY.CVPDE.59-2020.15,
  HLYZ.JDG.110-2018.1,
  HZ.Adv.332-2018.57,
  LLL.IMRNI.-2022.9114,
  LL.TAMS.373-2020.5833,
  LYZ.Adv.329-2018.85}.

In this paper we are concerned with the number of
solutions to the dual Minkowski problem \eqref{dMP}. 
When $q<0$, the uniqueness was proved by Zhao \cite{Zha.CVPDE.56-2017.18}.
When $q=0$, the uniqueness (up to a multiplicative constant) was obtained by
Alexandrov \cite{Ale.CRDASUN.35-1942.131}. 
When $q=n$, the solutions describe self-similar shrinking hypersurfaces for
Gauss curvature flows, whose uniqueness was proved by 
Abresch-Langer \cite{AL.JDG.23-1986.175},
Andrews \cite{And.Invent.138-1999.151, And.JAMS.16-2003.443},
Huang-Liu-Xu \cite{HLX.Adv.281-2015.906}, and
Brendle-Choi-Daskalopoulos \cite{BCD.Acta.219-2017.1} for $f\equiv1$;
by Dohmen-Giga \cite{DG.PJASAMS.70-1994.252}, and
Gage \cite{Gag.DMJ.72-1993.441} for even $f$ and $n=2$;
and nonuniqueness was found by
Yagisita \cite{Yag.CVPDE.26-2006.49} for some non-even $f$ and $n=2$.
When $q>2n$, the nonuniqueness of solutions was obtained for $f\equiv1$ by
Huang-Jiang \cite{HJ.JFA.277-2019.2209} and
Chen-Chen-Li \cite{CCL.AP.14-2021.689}. 

The remaining case when $0<q\leq 2n$ $(q\neq n)$ is more complicated.
As far as we know, there is only one result, which is about the uniqueness of
\emph{even} solutions for $0<q<n$ and $f\equiv1$; see
Chen-Huang-Zhao \cite{CHZ.MA.373-2019.953}.

The aim of this paper is to solve the remaining case for $f\equiv1$ in the
two-dimensional plane.
In this situation, the dual Minkowski problem \eqref{dMP} is reduced to the
following second-order ordinary differential equation:
\begin{equation} \label{dMP21}
  h (h'^2+h^2)^{\frac{q-2}{2}}(h''+h) = 1 \quad \text{ on } \ \Sy,
\end{equation}
where the unit circle $\Sy$ is parametrized by its arc length
$\theta\in[0,2\pi]$, and the superscript $'$ denotes $\frac{\dd}{\dd\theta}$.
As can be seen from the introduction in the previous paragraph, there is no
result on the uniqueness or nonuniqueness of solutions to Eq. \eqref{dMP21} when
$0<q\leq4$ $(q\neq2)$. 
Here in the present paper we provide a complete answer.

\begin{theorem}\label{thm1}
  \textup{(a)}
  When $0<q<1$, Eq. \eqref{dMP21} has exactly two solutions (up to rotation), of
  which one is the constant solution and the other is a positive smooth non-constant
  solution.

  \textup{(b)}
  When $1\leq q\leq 4$, Eq. \eqref{dMP21} has only one solution, which is just
  the constant solution.
\end{theorem}

The dual Minkowski problem is defined only among convex bodies containing the
origin in their interiors \cite{HLYZ.Acta.216-2016.325}, i.e., it requires a
solution to be positive.
For Eq. \eqref{dMP21}, such a solution is obviously smooth.
We note that Theorem \ref{thm1} is still true when a solution is required only
to be nonnegative and $C^2$; see Lemma \ref{lem005}.

As a by-product of proving the main theorem above, the following nonuniqueness
result for $q>4$ is obtained. 
\begin{theorem}\label{thm2}
  When $q>4$, Eq. \eqref{dMP21} admits at least $\lceil\! \sqrt{q} \,\rceil -2$
  non-constant positive smooth solutions (up to rotation), where $\lceil\!
  \sqrt{q} \,\rceil$ is the smallest integer that is greater than or equal to
  $\sqrt{q}$. 
\end{theorem}

It has already been known that Eq. \ref{dMP21} with $q>4$ admits at least one
non-constant positive smooth solution
\cite{HJ.JFA.277-2019.2209, CCL.AP.14-2021.689}.
Our Theorem \ref{thm2} shows that the number of non-constant solutions will increase
when $q$ becomes larger.

Now we can provide an application of these two theorems.
The $L_p$-Alexandrov problem, which is a natural extension of the Alexandrov
problem, was first studied by Huang-Lutwak-Yang-Zhang \cite{HLYZ.JDG.110-2018.1}.
For the two-dimensional case, when the given measure is the standard spherical
measure on $\Sy$, the $L_p$-Alexandrov problem is equivalent to solving the
following second-order ordinary differential equation:
\begin{equation} \label{LpAP21}
  h^{1-p} (h'^2+h^2)^{-1}(h''+h) = 1 \quad \text{ on } \ \Sy,
\end{equation}
where $p$ is a given real number.

\begin{theorem}\label{thm3}
  \textup{(a)}
  When $-1<p<0$, Eq. \eqref{LpAP21} has exactly two solutions (up to rotation),
  of which one is the constant solution and the other is a positive smooth non-constant
  solution.

  \textup{(b)}
  When $-4\leq p\leq -1$, Eq. \eqref{LpAP21} has only one solution, which is
  just the constant solution.

  \textup{(c)}
  When $p<-4$, Eq. \eqref{LpAP21} admits at least $\lceil\! \sqrt{-p} \,\rceil
  -2$ non-constant positive smooth solutions (up to rotation), where $\lceil\!
  \sqrt{-p} \,\rceil$ is the smallest integer that is greater than or equal to
  $\sqrt{-p}$.
\end{theorem}

One can verify that the duality between solutions of equations \eqref{dMP21} and
\eqref{LpAP21} via the concept of polar body is a one-to-one correspondence,
when $q=-p$; see e.g. \cite[Theorem 7.1]{CHZ.MA.373-2019.953}.
Therefore, Theorem \ref{thm3} follows directly from Theorems \ref{thm1} and \ref{thm2}.

For the reader's convenience, we here give the explicit expression of the duality.
Assume $h(\theta)$ is a solution to Eq. \eqref{dMP21}, then define its dual function
\begin{equation} \label{dual}
  \tilde{h}(\tilde{\theta}) =
  (h'^2+h^2)^{-\frac{1}{2}}(\theta)
  \  \text{ with } \ 
  \tilde{\theta}=\theta+\arctan\frac{h'(\theta)}{h(\theta)}.
\end{equation}
We can easily check that $\tilde{\theta}$ is strictly increasing with respect to
$\theta$, and that $\tilde{h}(\tilde{\theta})$ is a solution to Eq. \eqref{LpAP21}.
Conversely, one can construct $h$ from $\tilde{h}$ in the same way.

Similar to Lemma \ref{lem005}, one can easily check that Theorem \ref{thm3} still holds
when a solution is required only to be nonnegative and $C^2$.

Our main Theorem \ref{thm1} is actually to find out the number of $2\pi$-periodic
solutions to a second-order ordinary differential equation \eqref{dMP21}.
A commonly used method is considering solutions to Eq. \eqref{dMP21} with initial conditions,
and analyzing their period to see if it equals $2\pi$.
In Section \ref{sec2}, with some elementary analyses and computations, the
half-period can be expressed as an integral with parameters $\Theta_q(r)$.
Now we just need to analyze all possible values of this integral.
Generally speaking, it is not easy.
When $0<q\leq2$, it is achieved by proving monotonicity,
which is based on a monotonicity result of Andrews \cite{And.JAMS.16-2003.443};
see Section \ref{sec3}.
When $q>3$, one can easily see that $\Theta_q(r)$ is no longer monotonic.
Without monotonicity, analyzing all values of the integral is very difficult.
In Section \ref{sec4},
this difficulty is overcome by combining the theoretical analysis and numerical
estimation for $q=4$.
Then, one can solve all the cases for $1\leq q\leq4$.

After completing this paper, we noticed the paper \cite{LW},
some results of which coincided with this paper.
So let us make some remarks here.

The first time we submitted this paper for publication is June 4, 2022.
On September 29, 2022, the paper \cite{LW} was submitted to arXiv.
This was the first time we knew about the paper.
Then, we submitted our paper to arXiv on September 30, 2022.

The $L_p$ dual Minkowski problem, which is a generalization of Eq. \eqref{dMP21},
is considered in the paper \cite{LW}.
For a large range of $p,q$, uniqueness or nonuniqueness results are obtained.
The method there is also based on
the monotonicity result of Andrews and the duality \eqref{dual}.
Namely, the method and result in our Section \ref{sec3} are also obtained in \cite{LW}.

But it is worth mentioning that we still have the numerical estimation in
Section \ref{sec4}.
As mentioned above, the integral $\Theta_q(r)$ (corresponding to $p=0$) is not
monotonic for $3<q\leq4$.
Therefore, the paper \cite{LW} does not solve the uniqueness problem for this case.
While our paper have solved it by numerical estimation.
In fact, there also exists some other unsolved ranges of $p,q$ by \cite{LW},
precisely because the corresponding integral is not monotonic.
Now, an interesting question is whether our numerical estimation can be further
developed to solve these remaining unsolved cases.



\section{An integral with parameters}
\label{sec2}

In this section, we first give the following simple observation.

\begin{lemma}\label{lem005}
If $h$ is a nonnegative $C^2$ solution to Eq. \eqref{dMP21}, then $h$ is
positive and smooth. 
\end{lemma}

\begin{proof}
First note that $h$ is a support function of some convex body in the plane,
$h\not\equiv0$, implying $h$ should be positive at some interval. 

Now assume $h$ has at least one zero point.
Without loss of generality, one can assume that
\begin{equation*}
  h(0)=0 \text{ and } h(\theta)>0 \text{ for } \theta\in(0,\delta), 
\end{equation*}
where $\delta$ is a small positive constant.
We compute in $(0,\delta)$ as following:
\begin{equation*}
  \begin{split}
    \bigl[ (h'^2+h^2)^{\frac{q}{2}} \bigr]'
    &= q (h'^2+h^2)^{\frac{q-2}{2}}(h''+h)h' \\
    &= \frac{qh'}{h} \\
    &=(q\log h)',
  \end{split}
\end{equation*}
where Eq. \eqref{dMP21} is used for the second equality.
Thus, there exists some constant $E$ such that
\begin{equation}\label{eq:61}
  (h'^2+h^2)^{\frac{q}{2}} -q\log h=E
\end{equation}
holds for any $\theta\in(0,\delta)$.
Since $h$ attains its minimum at $\theta=0$, there is $h'(0)=0$. Recalling
$q>0$, we see the equality \eqref{eq:61} cannot hold when $\theta$ is
sufficiently close to $0$.
This contradiction shows that $h$ is positive, and then smooth.
\end{proof}

Obviously, $h\equiv1$ is the unique constant solution to Eq. \eqref{dMP21}.
In the rest of this paper, we consider only non-constant solutions.

Let $h$ be a non-constant solution, which is positive and smooth by Lemma
\ref{lem005}.
Then, equality \eqref{eq:61} holds on the whole $\Sy$.

\begin{lemma}\label{lem006}
  The zeros of $h'$ are isolated.
Moreover, each zero of $h'$ is either a minimum point or a maximum point. 
\end{lemma}

\begin{proof}{}
  Assume there exist distinct zeros
  $\set{\theta_0, \theta_1, \cdots, \theta_k, \cdots}$
  of $h'$, such that $\theta_k\to\theta_0$ as $k\to\infty$.
Then, we have by definition that $h''(\theta_0)=0$.
Recalling Eq. \eqref{dMP21}, there is $h(\theta_0)=1$.
But there is only one solution to Eq. \eqref{dMP21} satisfying $h(\theta_0)=1$ and
$h'(\theta_0)=0$, i.e., $h\equiv1$. This contradiction shows that any zero of
$h'$ must be isolated.

For simplicity, write
\begin{equation} \label{eq:62}
h_0=\min_{\theta\in\Sy}h(\theta), \quad h_1=\max_{\theta\in\Sy}h(\theta). 
\end{equation}
Let
\begin{equation*}
\phi(t)=t^q-q\log t, \quad \forall\, t>0.
\end{equation*}
Recalling the equality \eqref{eq:61}, we see
\begin{equation}\label{eq:63}
\phi(h_0)=\phi(h_1)=E.
\end{equation}
Observing that $\phi(t)$ is strictly decreasing in $(0,1)$, and strictly
increasing in $(1,\infty)$, one can see that $E>1$, $h_0<1<h_1$, and there are
no other points making $\phi(\cdot)=E$. 

Now if $\tilde{\theta}$ is a zero of $h'$, then $\phi(h(\tilde{\theta}))=E$ by
\eqref{eq:61}. Hence, $h(\tilde{\theta})$ must equal $h_0$ or $h_1$, implying
$\tilde{\theta}$ is either a minimum point or a maximum point.
\end{proof}

By virtue of Lemma \ref{lem006}, there are only finite minimum and maximum
points for $h$ on $\Sy$.
Fixed a minimum point $\theta_0$, let $\theta_1$ be one of the two nearest
maximum points. Without loss of generality, assume $\theta_0<\theta_1$.
Again by Lemma \ref{lem006}, $h'(\theta)\neq0$, implying $h'(\theta)>0$, when
$\theta\in(\theta_0,\theta_1)$.
Recalling \eqref{eq:61}, we have
\begin{equation*}
  h'(\theta)=\sqrt{\bigl(q\log h(\theta)+E\bigr)^{\frac{2}{q}}-h(\theta)^2},
  \quad \forall \theta\in(\theta_0,\theta_1).
\end{equation*}
Write the inverse function of $h$ in $(\theta_0,\theta_1)$ as
\begin{equation*}
\theta=h^{-1}(u), \quad \forall u\in(h_0,h_1),
\end{equation*}
where $h_0$ and $h_1$ are given as in \eqref{eq:62}.

Now the total curvature of the corresponding curve segment related to $h$ from
$\theta_0$ to $\theta_1$, denoted by $\Theta_q$, can be computed as: 
\begin{equation*}
  \begin{split}
    \Theta_q = \int_{\theta_0}^{\theta_1} \dd \theta 
    &= \int_{h_0}^{h_1}\! \frac{\dd u}{h'(h^{-1}(u))} \\
    &= \int_{h_0}^{h_1}\! \frac{\dd u}{\sqrt{(q\log u +E)^{2/q}-u^2}}.
  \end{split}
\end{equation*}
For further simplification, write $r=\frac{h_1}{h_0}>1$. Then
From \eqref{eq:63}, one has
\begin{equation}\label{eq:64}
  h_0^q=\frac{q\log r}{r^q-1}.
\end{equation}
Using the variable substitution $u=h_0\,x$, and
combining \eqref{eq:63} and \eqref{eq:64}, we obtain
\begin{equation} \label{Theta}
  \Theta_q = \Theta_q(r)
  = \int_1^r \!\frac{\dd x}{\sqrt{\bigl( 1+\frac{r^q-1}{\log r}\log x \bigr)^{\frac{2}{q}} -x^2}}.
\end{equation}
One can easily check that 
\begin{equation} \label{eq:65}
  1+\frac{r^q-1}{\log r}\log x  >x^q, \quad \forall x\in(1,r)
\end{equation}
holds for any $r>1$, and that $\Theta_q(r)$ can be defined for any $r>1$.

\begin{lemma}\label{lem007}
  For each $q>0$, the number of non-constant solutions (up to rotation) to Eq. \eqref{dMP21} is
  equal to the number of elements in the set
\begin{equation}\label{eq:66}
  \set{r>1 : \Theta_q(r)=\frac{\pi}{k} \text{ for some positive integer }k}.
\end{equation}
\end{lemma}

\begin{proof}
For a non-constant solution $h$ to Eq. \eqref{dMP21}, it has total curvature
$2\pi$, implying $\Theta_q(r)$ with $r=\frac{h_1}{h_0}$ must equal
$\frac{\pi}{k}$ for some positive integer $k$.
Different solutions have different minimum values $h_0$, and then have different
$r$-values by \eqref{eq:64}.

On the other hand, for any $r>1$ such that $\Theta_q(r)=\frac{\pi}{k}$ for some
positive integer $k$, let $h_0$ be determined by \eqref{eq:64}.
Note that $0<h_0<1$.
Then, the following second-order ordinary differential equation
\begin{equation*}
  h (h'^2+h^2)^{\frac{q-2}{2}}(h''+h) = 1 \quad \text{ in } \ \R 
\end{equation*}
with initial conditions
\begin{equation*}
h(0)=h_0, \quad h'(0)=0
\end{equation*}
has a unique positive smooth solution $h$, which is obviously non-constant.
Recalling Lemma \ref{lem006}, $h_0<1$, and equality \eqref{eq:63}, we see $h_0$
is the minimum value.
The maximum value of $h$, $h_1$, is also determined uniquely by \eqref{eq:63}.
Simple computations yield $h_1=rh_0$.
Therefore, the period of $h$ is
\begin{equation*}
  2\Theta_q\left(\frac{h_1}{h_0}\right)
  =2\Theta_q(r)
  =\frac{2\pi}{k}, 
\end{equation*}
which shows that $h$ is a solution to Eq. \eqref{dMP21}.
Thus, the conclusion of this lemma is true. 
\end{proof}

By virtue of Lemma \ref{lem007}, Theorems \ref{thm1} and \ref{thm2} are
equivalent to the following lemma.

\begin{lemma}\label{lem008}
  \textup{(a)}
  When $0<q<1$, the set defined in \eqref{eq:66} has only one element.

  \textup{(b)}
  When $1\leq q\leq 4$, the set defined in \eqref{eq:66} is empty.

  \textup{(c)}
  When $q>4$, the set defined in \eqref{eq:66} has at least
  $\lceil\! \sqrt{q} \,\rceil -2$ elements. 
\end{lemma}

In the rest of the paper, we will analyze $\Theta_q(r)$ to prove Lemma \ref{lem008}.

\section{The case $0<q\leq2$}
\label{sec3}

This section deals with the case when $0<q\leq2$ by proving that $\Theta_q(r)$ is
strictly decreasing with respect to $r$ in $(1,+\infty)$.

We first note that $\Theta_q(r)$ with $q>0$ is continuous at every $r\in(1,+\infty)$.
In fact, using the variable substitution $t=\frac{\log x}{\log r}$ for the
integral $\Theta_q(r)$ defined in \eqref{Theta}, we obtain 
\begin{equation} \label{eq:9}
\Theta_q(r) = \int_0^1 \!\frac{r^t\log r\dd t}{\sqrt{\bigl[ 1+(r^q-1)t \bigr]^{\frac{2}{q}} -r^{2t}}}.
\end{equation}
One can easily show that for any closed interval $[a,b]\subset(1,+\infty)$,
there exist positive constants $C$ and $\tilde{C}$ depending only on $a$, $b$
and $q$, such that
\begin{equation} \label{eq:67}
C t(1-t)\leq \bigl[ 1+(r^q-1)t \bigr]^{\frac{2}{q}} -r^{2t} \leq \tilde{C} t(1-t),
\quad \forall\, t\in(0,1) \text{ and } r\in[a,b].
\end{equation}
Then, the continuity of $\Theta_q(r)$ follows from the dominated convergence theorem.

Now we begin to consider the monotonicity of $\Theta_q(r)$ with respect to $r$.
Observing that the integrand in \eqref{eq:9} increases for some $t$ and
decreases for others when $r$ is near $1$, the expression in \eqref{eq:9} is not
suitable for directly deriving the monotonicity.

Inspired by \cite{And.JAMS.16-2003.443}, we consider the following variable substitution
\begin{equation}\label{xt}
x=\left( \bigl(r^{\frac{2q}{3}}-1\bigr)t+1 \right)^{\frac{3}{2q}}, \quad 0<t<1
\end{equation}
for $\Theta_q(r)$ in \eqref{Theta}.
By
\begin{equation*}
  \dd x=\frac{3}{2q}x^{1-\frac{2q}{3}}\bigl(r^{\frac{2q}{3}}-1\bigr)\dd t, 
\end{equation*}
we have that
\begin{equation} \label{eq:42}
  \begin{split}
    \Theta_q(r)
    &=
    \frac{3}{2q} \int_0^1 \!\frac{\bigl(r^{\frac{2q}{3}}-1\bigr)\dd t}{x^{\frac{2q}{3}}
      \sqrt{x^{-2}\bigl( 1+\frac{r^q-1}{\log r}\log x \bigr)^{\frac{2}{q}} -1}} \\
    &=
    \frac{3}{2q} \int_0^1 \!\frac{\dd t}{\sqrt{i_q(r,t)}},
  \end{split}
\end{equation}
where
\begin{equation}\label{eq:52}
  i_q(r,t):=
  \frac{x^{\frac{4q}{3}}}{\bigl(r^{\frac{2q}{3}}-1\bigr)^2}
  \left( x^{-2}\bigl( 1+\frac{r^q-1}{\log r}\log x \bigr)^{\frac{2}{q}} -1 \right).
\end{equation}


Recall Eq. \eqref{dMP21} with $q=2$ is reduced into the following equation
\begin{equation*}
  h (h''+h) = 1 \quad \text{ on } \ \Sy,
\end{equation*}
which was studied by Andrews in \cite{And.JAMS.16-2003.443}.
There, the fact
\begin{equation} \label{eq:68}
  \pd_r i_2(r,t)>0 \ \text{ for every }r>1 \text{ and }t\in(0,1)
\end{equation}
was proved.
For the reader’s convenience, we here list it as the following lemma.

\begin{lemma}[{\cite[Theorem 5.1]{And.JAMS.16-2003.443}}] \label{lemAnd}
Assume $t\in(0,1)$. Let $v=r^{\frac{4}{3}}t+1-t$, and define
\begin{equation*}
  J(r,t) = \frac{v^2}{\bigl( r^{\frac{4}{3}}-1 \bigr)^2} \left(
    v^{-\frac{3}{2}}\Bigl( 1+\frac{r^2-1}{\frac{4}{3}\log r}\log v \Bigr) -1 \right).
\end{equation*}
Then $\pd_r J(r,t)>0$ for every $r>1$ and $t\in(0,1)$.
\end{lemma}

Note that $i_2(r,t)\equiv J(r,t)$, and then Lemma \ref{lemAnd} is just \eqref{eq:68}.

However, it seems very difficult for the authors to carry out similar techniques
to directly prove the monotonicity of $i_q$ when $q\neq2$.
Fortunately, we have observed an interesting connection between the monotonicity
of the case $0<q<2$ and the case $q=2$.

In order to state this connection, we need to write $r=R^{\frac{3}{2q}}$ to
simplify the above expressions \eqref{xt}-\eqref{eq:52}. 
In fact, we have
\begin{equation} \label{wt}
  x=w^{\frac{3}{2q}} \text{ with } w=Rt+1-t,
\end{equation}
\begin{equation}\label{eq:53}
  I_q(R,t):=
  i_q(R^{\frac{3}{2q}},t) =
  \frac{w^2}{(R-1)^2} \left(
    w^{-\frac{3}{q}}\Bigl( 1+\frac{R^{\frac{3}{2}}-1}{\log R}\log w \Bigr)^{\frac{2}{q}} -1
  \right),
\end{equation}
and
\begin{equation} \label{tTheta}
  \widetilde{\Theta}_q(R) :=
  \Theta_q(R^{\frac{3}{2q}}) =
  \frac{3}{2q} \int_0^1 \!\frac{\dd t}{\sqrt{I_q(R,t)}}.
\end{equation}
Note that $I_q$ is much simpler than $i_q$ as expressions of $q$.

\begin{lemma}\label{lem003}
  For any fixed $t\in(0,1)$, if
\begin{equation*}
\pdR I_q(R,t)>0, \quad \forall R>1
\end{equation*}
holds for $q=2$, it also holds for $0<q<2$.
\end{lemma}

\begin{proof}
For simplicity, write
\begin{equation} \label{eq:54}
  F_q(R):=
  w^{-\frac{3}{q}}\Bigl( 1+\frac{R^{\frac{3}{2}}-1}{\log R}\log w \Bigr)^{\frac{2}{q}}.
\end{equation}
Note that $F_q(R)>1$ for any $R>1$, and $F_q(R)^q$ is independent of $q$, implying
\begin{equation*}
F_q(R)=F_2(R)^{\frac{2}{q}}, \quad \forall R>1.
\end{equation*}
We also write
\begin{equation}\label{eq:55}
G(R):=\frac{w^2}{(R-1)^2}=\left( t+\frac{1}{R-1} \right)^2.
\end{equation}
Then $G'(R)<0$ for every $R>1$.
Here and in the following proof of this lemma, the superscript $'$ denotes
$\frac{\dd}{\dd R}$, unless otherwise explicitly stated.

With \eqref{eq:54} and \eqref{eq:55}, the $I_q(R,t)$ defined in \eqref{eq:53}
reads
\begin{equation*}
  \begin{split}
    I_q(R,t)
    &=G(R)[F_q(R)-1] \\
    &=G(R)[F_2(R)^{\frac{2}{q}}-1].
  \end{split}
\end{equation*}
Therefore, 
\begin{equation}\label{eq:56}
  \pdR I_q = \bigl(F_2^{\frac{2}{q}}-1\bigr)G' +\frac{2}{q}F_2^{\frac{2}{q}-1}F_2'G.
\end{equation}
In particular,
\begin{equation*}
  \pdR I_2 = (F_2-1)G' +F_2'G,
\end{equation*}
namely $F_2'G = \pdR I_2 - (F_2-1)G'$. Inserting it into \eqref{eq:56}, we
obtain
\begin{equation}\label{eq:57}
  \begin{split}
    F_2^{1-\frac{2}{q}} \pdR I_q
    &=
    \frac{2}{q}\pdR I_2  
    +\bigl(F_2-F_2^{1-\frac{2}{q}}\bigr)G'
    -\frac{2}{q} (F_2-1)G' \\
    &=
    \frac{2}{q}\pdR I_2  
    -\left(
      F_2^{1-\frac{2}{q}} +\Bigl(\frac{2}{q}-1\Bigr)F_2 -\frac{2}{q}
    \right)G'.
  \end{split}
\end{equation}

Let
\begin{equation*}
  \phi(s)= s^{1-\frac{2}{q}} +\Bigl(\frac{2}{q}-1\Bigr)s -\frac{2}{q},
  \quad \forall s>0,
\end{equation*}
then
\begin{align*}
  \phi'(s)&= -\Bigl( \frac{2}{q}-1 \Bigr)s^{-\frac{2}{q}} +\frac{2}{q}-1, \\
  \phi''(s)&= \frac{2}{q} \Bigl( \frac{2}{q}-1 \Bigr) s^{-\frac{2}{q}-1}.
\end{align*}
When $0<q<2$, $\phi''(s)>0$. By virtue of $\phi'(1)=0$, we see that
\begin{equation*}
\phi(s)>\phi(1)=0, \quad \forall\, 0<s<1 \text{ or } s>1.
\end{equation*}
Since $F_2>1$ for any $R>1$, then $\phi(F_2)>0$, namely
\begin{equation}\label{eq:58}
  F_2^{1-\frac{2}{q}} +\Bigl(\frac{2}{q}-1\Bigr)F_2 -\frac{2}{q} >0, \quad \forall R>1. 
\end{equation}
Combining this inequality \eqref{eq:58} and the fact $G'<0$, we see from
\eqref{eq:57} that
\begin{equation} \label{eq:59}
  F_2^{1-\frac{2}{q}} \pdR I_q > \frac{2}{q}\pdR I_2.
\end{equation}
Hence, if $\pdR I_2>0$ for every $R>1$, then $\pdR I_q>0$ for every $R>1$ and
$0<q<2$. 
This means the lemma is true.
\end{proof}

Obviously, by \eqref{eq:53}, \eqref{eq:68} is equivalent to
\begin{equation}\label{eq:60}
  \pdR I_2(R,t)>0, \quad \forall R>1 \text{ and } t\in(0,1),
\end{equation}
which together with Lemma \ref{lem003} implies the following lemma.

\begin{lemma}\label{lem004}
  For every $0<q\leq2$, there is
\begin{equation*}
\pdR I_q(R,t)>0, \quad \forall R>1 \text{ and } t\in(0,1).
\end{equation*}
As a result, \eqref{tTheta} implies that $\widetilde{\Theta}_q(R)$ is strictly
decreasing in $R\in(1,+\infty)$, i.e., $\Theta_q(r)$ is strictly
decreasing in $r\in(1,+\infty)$. 
\end{lemma}

Now to find out all values of $\Theta_q(r)$ for $0<q\leq2$, it only needs to
compute the values when $r$ approaches one or infinity.

\begin{lemma}\label{lem009}
  When $q>0$, we have
\begin{equation*}
  \lim_{r\to1} \Theta_q(r) =\frac{\pi}{\sqrt{q}}
  \quad \text{ and } \quad
  \lim_{r\to\infty} \Theta_q(r) =\frac{\pi}{2}. 
\end{equation*}
\end{lemma}

\begin{proof}{}
The proof involves only basic calculations.
However, one should be careful when taking a limit inside an integral.
Here we provide the main steps.

(a) The case when $r\to1$.
Using the Taylor expansion, one can easily check that
\begin{equation*}
  \bigl[ 1+(r^q-1)t \bigr]^{\frac{2}{q}} -r^{2t}
  =qt(1-t)(r-1)^2[1+\alpha(r,t)],
\end{equation*}
where $\alpha(r,t)$ converges to zero uniformly for $t\in(0,1)$ when $r\to1$.
Applying the dominated convergence theorem to \eqref{eq:9}, we have
\begin{equation*}
\lim_{r\to1}\Theta_q(r) = \int_0^1 \!\frac{\dd t}{\sqrt{qt(1-t)}} =\frac{\pi}{\sqrt{q}}.
\end{equation*}

(b) The case when $r\to\infty$.
We need more efforts to deal with this case. 
For simplicity, let
\begin{equation*}
a=\frac{1}{2}\min\set{q,1}\in(0,\tfrac{1}{2}]
\ \text{ and } \ 
b=\frac{1}{\sqrt{2}}a^{\frac{1}{q}}. 
\end{equation*}
We assume $r>2b^{\frac{1}{a-1}}$, and write $\Theta_q$ in \eqref{Theta} into three parts:
\begin{equation} \label{eq:16}
  \begin{split}
    \Theta_q(r)
    &=\biggl( \int_{1}^{r^a} +\int_{r^a}^{br} +\int_{br}^r \biggr)
    \!\frac{\dd x}{\sqrt{\bigl( 1+\frac{r^q-1}{\log r}\log x \bigr)^{\frac{2}{q}} -x^2}} \\
    &=: I(r) +\II(r) +\III(r).
  \end{split}
\end{equation}

For $I$, let
\begin{equation*}
  \phi(x)=\Bigl( 1+\frac{r^q-1}{\log r}\log x \Bigr)^{\frac{2}{q}} -x^2,
  \quad \forall x\in(1,r^a).
\end{equation*}
With inequality \eqref{eq:65}, we have
\begin{equation*}
  \begin{split}
    \phi'(x)
    &= \frac{2}{q} \Bigl( 1+\frac{r^q-1}{\log r}\log x \Bigr)^{\frac{2}{q}-1} \frac{r^q-1}{x\log r} -2x \\
    &\ge \frac{2(r^q-1)}{q\log r} \Bigl( 1+\frac{r^q-1}{\log r}\log x \Bigr)^{\frac{1}{q}-1} -2x \\
    &\ge \begin{cases}
      \frac{r^q}{q\log r}-2r^a, & \text{if } q\le1 \\
      \frac{r^q}{q\log r} (r^q)^{\frac{1}{q}-1}-2r^a, & \text{if }q>1
    \end{cases} \\
    &\ge \frac{r^{2a}}{2q\log r},
  \end{split}
\end{equation*}
where the last inequality holds when $r$ is sufficiently large. Hence, 
\begin{equation}\label{eq:18}
  \begin{split}
    I(r)
    &=\int_1^{r^a} \!\frac{\dd x}{\sqrt{\phi'(\xi)(x-1)}} \\
    &\le \sqrt{\frac{2q\log r}{r^{2a}}} \int_1^{r^a} \!\frac{\dd x}{\sqrt{x-1}} \\
    &= \sqrt{\frac{2q\log r}{r^{2a}}} \cdot 2\sqrt{r^a-1} \\
    &\to 0 \text{ as } r\to\infty.
  \end{split}
\end{equation}

For $\II$, using the variable substitution $t=rx^{-1}$, we have
\begin{equation} \label{eq:17}
  \II(r)
  = \int_{\frac{1}{b}}^{r^{1-a}} \!\frac{t^{-2}\dd t}{
    \sqrt{\bigl( 1-\frac{1-r^{-q}}{\log r}\log t \bigr)^{\frac{2}{q}} -t^{-2}}}.
\end{equation}
Since for every $t\in[\tfrac{1}{b},r^{1-a}]$, there is
\begin{equation*}
  1-\frac{1-r^{-q}}{\log r}\log t
  \geq 1-(1-r^{-q})(1-a) >a,
\end{equation*}
implying
\begin{equation*}
  \Bigl( 1-\frac{1-r^{-q}}{\log r}\log t \Bigr)^{\frac{2}{q}} -t^{-2}
  > a^{\frac{2}{q}} -b^2 =b^2.
\end{equation*}
Then, we see that
\begin{equation*}
  \II(r)
  \leq \int_{\frac{1}{b}}^{r^{1-a}} \!\frac{1}{b}t^{-2}\dd t
  \leq \int_{\frac{1}{b}}^{\infty} \!\frac{1}{b}t^{-2}\dd t.
\end{equation*}
Now one can apply the dominated convergence theorem to \eqref{eq:17} to obtain
\begin{equation} \label{eq:19}
  \lim_{r\to\infty} \II(r)
  = \int_{\frac{1}{b}}^{\infty} \!\frac{t^{-2}\dd t}{\sqrt{1 -t^{-2}}}
  = \int_{0}^{b} \!\frac{\dd \tau}{\sqrt{1 -\tau^2}}.
\end{equation}

For $\III$, using the variable substitution $t=r^{-1}x$, we have
\begin{equation} \label{eq:20}
  \III(r)
  = \int_{b}^{1} \!\frac{\dd t}{
    \sqrt{\bigl( 1+\frac{1-r^{-q}}{\log r}\log t \bigr)^{\frac{2}{q}} -t^2}}.
\end{equation}
Assume $r>b^{-2}$, then 
\begin{equation*}
  1+\frac{1-r^{-q}}{\log r}\log t \in \Bigl(\frac{1}{2},1\Bigr),
  \quad \forall\, t\in[b,1). 
\end{equation*}
Let $\phi(t)=\bigl( 1+\frac{1-r^{-q}}{\log r}\log t \bigr)^{\frac{1}{q}} -t$.
Then $\phi(1)=0$ and
\begin{equation*}
  \begin{split}
    -\phi'(t)
    &= -\frac{1}{q}\Bigl( 1+\frac{1-r^{-q}}{\log r}\log t \Bigr)^{\frac{1}{q}-1}
    \cdot \frac{1-r^{-q}}{t\log r}  +1 \\
    &\ge -\frac{C_q}{\log r} +1 \\
    &\ge \frac{1}{4} \text{ when $r$ is very large.}
  \end{split}
\end{equation*}
Therefore, $\III$ can be estimated as following:
\begin{equation*}
  \begin{split}
    \III(r)
    &\le \int_{b}^{1} \!\frac{\dd t}{\sqrt{t \phi(t)}} \\
    &= \int_{b}^{1} \!\frac{\dd t}{\sqrt{t [-\phi'(\xi)](1-t)}} \\
    &\le \int_{b}^{1} \!\frac{2\dd t}{\sqrt{b (1-t)}},
  \end{split}
\end{equation*}
which means that the dominated convergence theorem can be applied to
\eqref{eq:20}.
Thus, we obtain
\begin{equation} \label{eq:21}
  \lim_{r\to\infty}\III(r) = \int_{b}^{1} \!\frac{\dd t}{\sqrt{1 -t^2}}.
\end{equation}

Now combining \eqref{eq:16}, \eqref{eq:18}, \eqref{eq:19} and \eqref{eq:21}, we obtain
\begin{equation*} 
  \lim_{r\to\infty}\Theta_q(r)
  = \int_{0}^{1} \!\frac{\dd t}{\sqrt{1 -t^2}}
  =\frac{\pi}{2}.
\end{equation*}
The proof of this lemma is completed.
\end{proof}

On account of Lemmas \ref{lem004} and \ref{lem009}, we see Lemma \ref{lem008}
holds for the case when $0<q\leq2$.


\section{The case $1\leq q\leq4$}
\label{sec4}

In this section, we mainly deal with the case when $1\leq q\leq4$, as well as the
case when $q>4$.

Note that when $q>3$, $\Theta_q(r)$ is no longer monotonic with respect to $r$
in $(1,+\infty)$. 
In fact, by virtue of Lemmas \ref{lem009} and \ref{lem010}, there exists a
sufficiently large $r$ such that $\frac{\dd}{\dd r}\Theta_q(r)<0$.
One can also check that $\frac{\dd}{\dd r}\Theta_q(r)>0$ for some $r$ near $1$.

Due to lack of monotonicity, we need to develop some new methods to find out the
exact number of elements in the set \eqref{eq:66}. 
By combining the theoretical analysis and numerical estimation, we will prove
for $1\leq q\leq4$ that 
\begin{equation} \label{eq:10}
  \frac{\pi}{2}<\Theta_q(r)<\pi
  \ \text{ for every } r>1.
\end{equation}

We begin the proof with the following observation.

\begin{lemma} \label{lem01}
  For any $r>1$ and $0<q_1<q_2$, there is
  \begin{equation*}
    \Theta_{q_2}(r)<\Theta_{q_1}(r).
  \end{equation*}
\end{lemma}

\begin{proof}{}
Given any $r>1$, consider
\begin{equation*}
\phi(q)=\bigl[ 1+(r^q-1)a \bigr]^{\frac{2}{q}}, \quad 0<a<1,
\end{equation*}
then we have
\begin{equation}\label{eq:12}
  \frac{\phi'(q)}{\phi(q)}
  =\frac{2}{q} \left(\frac{a r^q \log r}{1+(r^q-1)a} -\frac{\log[ 1+(r^q-1)a ]}{q}\right).
\end{equation}
To prove $\phi'(q)>0$, we introduce the auxiliary function
\begin{equation*}
f(a)=(a+b)\log\left(1+\frac{b}{a}\right), \quad a>0, \  b>0.
\end{equation*}
Since
\begin{equation*}
f'(a)=\log\left(1+\frac{b}{a}\right)-\frac{b}{a}<0,
\end{equation*}
there is $f(a)>f(1)$ for $0<a<1$, namely,
\begin{equation} \label{eq:11}
  (a+b)\log\left(1+\frac{b}{a}\right)
  >
  (1+b)\log\left(1+b\right),
  \quad
  \forall 0<a<1,\ b>0.
\end{equation}
Replacing $b$ by $(r^q-1)a$ in the inequality \eqref{eq:11}, we obtain
\begin{equation*} 
  ar^q\log r^q
  >
  [1+(r^q-1)a]\log[1+(r^q-1)a],
  \quad
  \forall 0<a<1,
\end{equation*}
namely
\begin{equation*} 
  \frac{ar^q\log r}{1+(r^q-1)a}
  >
  \frac{\log[1+(r^q-1)a]}{q},
  \quad
  \forall 0<a<1.
\end{equation*}
Recalling \eqref{eq:12}, we see that $\phi'(q)>0$.
Therefore,
\begin{equation*}
\bigl[ 1+(r^{q_1}-1)a \bigr]^{\frac{2}{q_1}}
<
\bigl[ 1+(r^{q_2}-1)a \bigr]^{\frac{2}{q_2}}, \quad \forall 0<a<1.
\end{equation*}
For any $1<x<r$, replacing $a$ by $\frac{\log x}{\log r}$, we have
\begin{equation*}
\Bigl( 1+\frac{r^{q_1}-1}{\log r}\log x \Bigr)^{\frac{2}{q_1}}
<
\Bigl( 1+\frac{r^{q_2}-1}{\log r}\log x \Bigr)^{\frac{2}{q_2}},
\end{equation*}
which together with the definition of $\Theta_q(r)$ given in \eqref{Theta}
implies that $\Theta_{q_1}(r)>\Theta_{q_2}(r)$.
\end{proof}

On account of this lemma, we have for every $q\geq1$ that
\begin{equation} \label{eq:23}
  \Theta_q(r)\leq\Theta_1(r)<\pi, \quad \forall r>1.
\end{equation}
Moreover, in order to prove the first equality of \eqref{eq:10}, it only needs
to prove
\begin{equation} \label{eq:69}
  \Theta_4(r)> \frac{\pi}{2},
  \quad\forall r>1.
\end{equation}
This is the main aim in the rest of this section.

\begin{lemma} \label{lem010}
  For each $q>2$, there exists some $\tilde{r}\in(1,+\infty)$, such that
  \begin{equation}\label{eq:25}
    \Theta_q(r)>\frac{\pi}{2}, \quad \forall r\geq\tilde{r}.
  \end{equation}
  In fact, $\tilde{r}$ can be chosen as the unique root of
  \begin{equation} \label{eq:28}
    r^q-qr^{q-1}\log r-1 =0, \quad r>1.
  \end{equation}
  In particular, for $q=4$, there is $8.61<\tilde{r}<8.62$, implying
  \begin{equation}\label{eq:70}
    \Theta_4(r)>\frac{\pi}{2}, \quad \forall r\geq8.62.
  \end{equation}
\end{lemma}

\begin{proof}{}
Let
\begin{equation*}
  f(r)= r^q-qr^{q-1}\log r-1, \quad r>1.
\end{equation*}  
Then
\begin{equation*}
  f'(r)=qr^{q-2}[r-1-(q-1)\log r] =qr^{q-2}\phi(r).
\end{equation*}
Since $\phi'(r)=1-\frac{q-1}{r}$, $\phi(r)$ strictly decreases in $(1,q-1)$ and
strictly increases in $(q-1,+\infty)$.
Observing $\phi(1)=0$ and $\phi(+\infty)=+\infty$, there exists some
$r_1\in(q-1,+\infty)$, such that $\phi(r)<0$ for $r\in(1,r_1)$, and $\phi(r)>0$
for $r\in(r_1,+\infty)$.
Therefore, $f(r)$ strictly decreases in $(1,r_1)$ and strictly increases in
$(r_1,+\infty)$. Noting $f(1)=0$ and $f(+\infty)=+\infty$, we see that
$f(r)$ has a unique real zero, denoted by $\tilde{r}$, in $(1,+\infty)$, and
that
\begin{equation}\label{eq:29}
  r^q-qr^{q-1}\log r-1
\begin{cases}
  <0, &\forall r\in(1,\tilde{r}), \\
  >0, &\forall r\in(\tilde{r},+\infty).
\end{cases}
\end{equation}

For any given $r\geq \tilde{r}$, we want to prove the following inequality:
\begin{equation}\label{eq:26}
\left( 1+\frac{r^q-1}{\log r}\log x \right)^{\frac{2}{q}} -x^2
<(r-x)(r+x-2),
\quad \forall x\in(1,r),
\end{equation}
which is equivalent to the inequality
\begin{equation}\label{eq:27}
  1+\frac{r^q-1}{\log r}\log x 
  <
  \left(r^2-2 r+2 x\right)^{\frac{q}{2}},
  \quad \forall x\in(1,r).
\end{equation}

Consider the auxiliary function
\begin{equation*}
  g(x)=
  \left(r^2-2 r+2 x\right)^{\frac{q}{2}}
  -1-\frac{r^q-1}{\log r}\log x,
  \quad x>1.
\end{equation*}
Obviously, for $q>2$, $g''(x)>0$ for any $x>1$, and $g(r)=0$.
If $g'(r)\leq0$, then 
\begin{equation*}
g'(x)<g'(r)\leq0, \quad \forall x\in(1,r),
\end{equation*}
implying that
\begin{equation*}
g(x)>g(r)=0, \quad \forall x\in(1,r),
\end{equation*}
which is just the inequality \eqref{eq:27}, namely, \eqref{eq:26}.
Now we solve the assumption $g'(r)\leq0$.
Since
\begin{equation*}
  g'(x)= q \left(r^2-2 r+2 x\right)^{\frac{q}{2}-1}-\frac{r^q-1}{x \log r},
\end{equation*}
$g'(r)\leq0$ means
\begin{equation*}
q r^{q-2}-\frac{r^q-1}{r \log r}\leq0,
\end{equation*}
namely
\begin{equation*}
r^q-qr^{q-1}\log r-1\geq0,
\end{equation*}
which is equivalent to $r\geq\tilde{r}$ by virtue of \eqref{eq:29}.
Thus, we have proved the inequality \eqref{eq:26}.

Now we are ready to prove \eqref{eq:25} of this lemma.
In fact, combining \eqref{Theta} and \eqref{eq:26}, we have for any $r\geq\tilde{r}$ that
\begin{equation*}
  \begin{split}
    \Theta_q(r)
    &>
    \int_1^r \!\frac{\dd x}{\sqrt{(r-x)(r+x-2)}} \\
    &=
    \int_1^r \!\frac{\dd (x-1)}{\sqrt{(r-x)(r+x-2)}} \\
    &= \int_0^1 \!\frac{\dd t}{\sqrt{(1-t)(1+t)}} \\
    &= \frac{\pi}{2},
  \end{split}
\end{equation*}
where the variable substitution
\begin{equation*}
  t=\frac{x-1}{r-1},
  \quad
  1-t=\frac{r-x}{r-1}, 
  \quad
  1+t=\frac{r+x-2}{r-1}
\end{equation*}
has been used for the second equality.
The conclusion \eqref{eq:25} of this lemma is true. 
\end{proof}

Computing the Taylor expansion of $\Theta_4(r)$ about the point $r=1$, one can
easily see that
\begin{equation*}
  \Theta_4(r)>\frac{\pi}{2}, \quad \forall 1<r\leq\bar{r}
\end{equation*}
holds for some $\bar{r}$ which is close to $1$.
Here, we need to carry out a delicate Taylor expansion to obtain a specific value of $\bar{r}$.

\begin{lemma}\label{lem011}
$\Theta_4(r)>\frac{\pi}{2}$ holds for any $r\in(1,1.22]$.
\end{lemma}

\begin{proof}{}
Recalling \eqref{eq:9}, we have
\begin{equation} \label{eq:31}
\Theta_4(r) =\int_0^1 \!\frac{r^t\log r\dd t}{\sqrt{\bigl[ 1+(r^4-1)t \bigr]^{\frac{1}{2}} -r^{2t}}}.
\end{equation}
For convenience, we write $\delta=r-1$ and assume $0<\delta<\frac{3}{10}$ in the
following proof of this lemma.

Consider first the Taylor expansion of 
$\phi(r)=A(r)^{\frac{1}{2}} -r^{2t}$ with $A(r)=1-t+r^4t$.
By some lengthy but simple computations, we have that
\begin{equation}\label{eq:32}
  \begin{split}
\phi(r)
&=4t(1-t)\delta^2 \biggl(  
1
+\frac{1}{3} (1-2 t) \delta
+\frac{1}{12} (32 t^2-28 t+3) \delta^2 \\
&\hskip7.23em
-\frac{1}{30} (208 t^3-232 t^2+43 t+3) \delta^3
+P_0(r_t)\delta^4
\biggr),
  \end{split}
\end{equation}
where
\begin{multline*}
  P_0(r)
  = \frac{\phi^{(6)}(r)}{6!\cdot 4t(1-t)} 
  = \frac{1}{360} (t-2) (2 t-5) (2 t-3) (2 t-1) r^{2 t-6} \\
  +\frac{7}{8}A^{-\frac{11}{2}}  r^2 t \,(r^4 t+t-1)
  \left[1-2(5 r^4+1) t+(r^8+10 r^4+1) t^2\right],
\end{multline*}
and $r_t\in(1,r)$ depending on $r$ and $t$.

We want to obtain an upper bound of $P_0$ when $t\in(0,1)$ and $r\in(1,\frac{13}{10})$.
However, it is not easy to find its maximum directly, due to its complicated expression.
We instead estimate its first and second items respectively to provide an upper
bound of $P_0$.
In fact, for the first item, we have
\begin{equation*}
  \begin{split}
    &\sup_{0<t<1, \,1<r<\frac{13}{10}}  \frac{1}{360} (t-2) (2 t-5) (2 t-3) (2 t-1) r^{2 t-6}  \\
    &\hskip1.9em \leq  
    \sup_{0<t<\frac{1}{2}, \,1<r<\frac{13}{10}} \frac{1}{360} (t-2) (2 t-5) (2 t-3) (2 t-1) r^{2 t-6} \\
    &\hskip1.9em \leq  
    \sup_{0<t<\frac{1}{2}} \frac{1}{360} (t-2) (2 t-5) (2 t-3) (2 t-1)  \\
    &\hskip1.9em = \frac{1}{12}.
  \end{split}
\end{equation*}
For the second item, given a $r\in(1,\frac{13}{10})$, it is a function of
$t\in(0,1)$. One can check that it attains its supremum $\frac{7}{8}r^{-8}$ at $t=1$.
Then, the supremum of the second item is $\frac{7}{8}$.
Therefore, when $r\in(1,\frac{13}{10})$, we have
\begin{equation*}
P_0(r)\leq\frac{1}{12}+\frac{7}{8}=\frac{23}{24}.
\end{equation*}
Inserting it into \eqref{eq:32}, we obtain
\begin{equation*}
\phi(r) \leq 4t(1-t)\delta^2 (1+X), 
\end{equation*}
where 
\begin{equation*}
    X=
\frac{1-2t}{3} \delta
+\frac{1}{12} (32 t^2-28 t+3) \delta^2
-\frac{1}{30} (208 t^3-232 t^2+43 t+3) \delta^3
+\frac{23}{24}\delta^4
\end{equation*}
is a polynomial of $t$ and $\delta$.

Then, we have the following estimates:
\begin{equation}\label{eq:33}
  \begin{split}
\phi(r)^{-\frac{1}{2}}
&\geq
\frac{1}{\sqrt{4t(1-t)}}\delta^{-1} \left( 1+X \right)^{-\frac{1}{2}} \\
&\geq
\frac{1}{\sqrt{4t(1-t)}}\delta^{-1} \left( 1-\frac{1}{2}X+\frac{3}{8}X^2-\frac{5}{16}X^3 \right) \\
&=
\frac{1}{\sqrt{4t(1-t)}}\delta^{-1} \biggl( 
1
+\frac{2t-1}{6} \delta
-\frac{1}{12} (14 t^2-12 t+1) \delta^2
\\
&\hskip8.4em
+\frac{1}{1080}(2404 t^3-2346 t^2+84 t+109)\delta^3
+P_1 \delta^4 
\biggr),
\end{split}
\end{equation}
where $P_1$ is a polynomial of $t$ and $\delta$ with rational coefficients.
By a numerical computation~
\footnote{For a polynomial,
  given any $\epsilon>0$, we can always calculate an approximate minimum value
  with an error of less than $\epsilon$.
  Then, the approximate minimum value minus $\epsilon$ is a strictly correct
  lower bound.},
one can see that
\begin{equation*}
\inf_{0<t<1, \,0<\delta<\frac{3}{10}} P_1> -\frac{54}{100}.
\end{equation*}
Note that the right hand side of \eqref{eq:33} is always positive when replacing
$P_1$ by this lower bound.

Next, we compute the Taylor expansion of $r^t\log r$.
It is easy to check that
\begin{equation}\label{eq:39}
  \begin{split}
    r^t\log r
    &=
    \delta \biggl(
    1
    +\frac{2t-1}{2} \delta
    +\frac{1}{6} (3 t^2-6 t+2)\delta^2
    \\
&\hskip3.11em
    +\frac{1}{12} (2 t^3-9 t^2+11 t-3)\delta^3
    +P_2(r_t)\delta^4
    \biggr),
  \end{split}
\end{equation}
where $r_t\in(1,r)$ depending on $r$ and $t$, and
\begin{equation*}
  \begin{split}
    P_2(r)
    &= \frac{1}{5!}\cdot \frac{\dd^5}{\dd r^5} (r^t\log r) \\
    &= \frac{r^{t-5}}{120} \left(5 t^4-40 t^3+105 t^2-100 t+24 +(t-4) (t-3) (t-2) (t-1) t \log r \right) \\
    &= \frac{r^{t-5}}{120} \left[5 (t-2)^4-15 (t-2)^2+4 +(t-4) (t-3) (t-2) (t-1) t \log r \right] \\
    &\geq \frac{r^{t-5}}{120} \left( -\frac{29}{4} \right) \\
    &\geq -\frac{29}{480}.
  \end{split}
\end{equation*}
Note that the right hand side of \eqref{eq:39} is always positive when replacing
$P_2$ by its lower bound.

Now multiplying the two inequalities \eqref{eq:33} and \eqref{eq:39}, and
noting the lower bounds of $P_1$ and $P_2$, we obtain that
\begin{equation}\label{eq:41}
  \begin{split}
\frac{r^t\log r}{\sqrt{\phi(r)}}
    &\geq \frac{1}{\sqrt{4t(1-t)}} 
    \biggl(
    1
    +\frac{2}{3} (2 t-1) \delta
    +\frac{1}{3} (1-t-t^2)\delta^2 \\
    &\hskip7.61em
    +\frac{1}{135} (188 t^3-237 t^2+93 t-22)\delta^3
    +P_3\delta^4
    \biggr),
  \end{split}
\end{equation}
where $P_3$ is a polynomial of $t$ and $\delta$ with rational coefficients.
By a numerical computation, we can obtain that
\begin{equation*}
\inf_{0<t<1, \,0<\delta<\frac{3}{10}} P_3> -\frac{66}{100}.
\end{equation*}
Inserting this lower bound into \eqref{eq:41}, then integrating both sides
with respect to $t$, and noting 
\begin{equation*}
    \int_0^1 \!\frac{t^m\dd t}{\sqrt{4t(1-t)}}
    =\frac{\pi (2 m-1)!!}{2^{m+1} m!},
    \quad m=0,1,2,3,\cdots,
\end{equation*}
we have that
\begin{equation*}
  \begin{split}
    \Theta_4(r)
    &=\int_0^1 \frac{r^t\log r}{\sqrt{\phi(r)}} \dd t \\
    &> \frac{\pi}{2}
    +\frac{\pi}{48} \delta ^2
    -\frac{\pi}{48}\delta ^3
    -\frac{33\pi}{100} \delta ^4,
  \end{split}
\end{equation*}
implying $\Theta_4(r)>\frac{\pi}{2}$ when $0<\delta\leq\frac{22}{100}$.
The proof of this lemma is completed. 
\end{proof}

For the remaining case when $r\in(1.22,8.62)$, it seems very hard for the
authors to use pure theoretical analysis to prove $\Theta_4(r)>\frac{\pi}{2}$. 
We here develop a numerical method to achieve the proof.

Recalling the definition of $\Theta_q(r)$ given in \eqref{Theta}, we have
\begin{equation} \label{Theta4}
\Theta_4(r) = \int_1^r \!\frac{\dd x}{\sqrt{\bigl( 1+\frac{r^4-1}{\log r}\log x \bigr)^{\frac{1}{2}} -x^2}}.
\end{equation}

\begin{lemma}\label{lem001}
The lower and upper bounds of the integral $\Theta_4(r)$ can be estimated by some
elementary functions of $r$.
In fact, for any integer $k\geq2$, we can construct two elementary functions
$\lambda_k(r)$ and $\Lambda_k(r)$ such that
\begin{equation} \label{eq:37}
  \lambda_k(r)<\Theta_4(r)<\Lambda_k(r), \quad \forall r>1.
\end{equation}
\end{lemma}

\begin{proof}{}
Let
\begin{equation} \label{fx}
  f(x)=\left( 1+\frac{r^4-1}{\log r}\log x \right)^{\frac{1}{2}}, \quad x\in[1,r],
\end{equation}
which is obviously a strictly concave function.
For the given integer $k$, let $1=x_0<x_1<\cdots<x_{k-1}<x_k=r$ be a partition that
divides $[1,r]$ into $k$ sub-intervals of equal length $\frac{r-1}{k}$, namely
\begin{equation*}
x_i=1+\frac{r-1}{k}i \ \ \text{ for } i=0,1,2,\cdots,k.
\end{equation*}

We first derive the lower bound of $\Theta_4$.
Let
\begin{equation} \label{gx}
g(x)=\frac{r^4-1}{2\log r} \cdot \frac{1}{xf(x)}
\end{equation}
be the derivative of $f(x)$.
On each sub-interval $[x_{i-1},x_i]$ for $i=2,\cdots,k$, by the strict concavity
of $f$, there is
\begin{equation*}
0\leq f(x)-x^2\leq f(x_i)+g(x_i)(x-x_i)-x^2,
\end{equation*}
where the equality of the second inequality holds only for $x=x_i$.
Similarly, for the sub-interval $[1,x_1]$, we have
\begin{equation*}
0\leq f(x)-x^2\leq f(1)+g(1)(x-1)-x^2.
\end{equation*}
Then
\begin{equation*}
  \begin{split}
    \Theta_4(r)
    &=\int_{1}^{x_1} \!\frac{\dd x}{\sqrt{f(x) -x^2}}
    +\sum_{i=2}^k \int_{x_{i-1}}^{x_i} \!\frac{\dd x}{\sqrt{f(x) -x^2}} \\
    &>
    \int_{1}^{x_1} \!\frac{\dd x}{\sqrt{1+g(1)(x-1)-x^2}}
    +\sum_{i=2}^k \int_{x_{i-1}}^{x_i} \!\frac{\dd x}{\sqrt{f(x_i)+g(x_i)(x-x_i)-x^2}} \\
    &=: \lambda_k(r).
  \end{split}
\end{equation*}
By the formula \eqref{eq:40} in the following Lemma \ref{lem000}, we can obtain
the explicit expression of $\lambda_k(r)$:
\begin{multline} \label{eq:44}
  \lambda_k(r)=
  2\arcsin \left(
    \frac{1}{2} +\frac{1+\frac{r-1}{k}-\frac{1}{2}g(1)}{\sqrt{g(1)^2+4[1-g(1)]}}
  \right)^{\frac{1}{2}} \\
  +\sum_{i=2}^k \tilde{\lambda}\left( 1+\frac{r-1}{k}(i-1), 1+\frac{r-1}{k}i \right),
\end{multline}
where
\begin{multline}\label{eq:43}
  \tilde{\lambda}(y,z) =
  2\arcsin \left(
    \frac{1}{2} +\frac{z-\frac{1}{2}g(z)}{\sqrt{g(z)^2+4[f(z)-g(z)z]}}
  \right)^{\frac{1}{2}} \\
  -2\arcsin \left(
    \frac{1}{2} +\frac{y-\frac{1}{2}g(z)}{\sqrt{g(z)^2+4[f(z)-g(z)z]}}
  \right)^{\frac{1}{2}}.
\end{multline}

Now we derive the upper bound of $\Theta_4$. 
On each sub-interval $[x_{i-1},x_i]$ for $i=1,\cdots,k$, by the strict concavity
again, we have
\begin{equation*}
  f(x)-x^2\geq f(x_i)+\gamma(x_{i-1},x_i)(x-x_i)-x^2\geq0,
\end{equation*}
where
\begin{equation} \label{eq:45}
  \gamma(y,z)=\frac{f(z)-f(y)}{z-y},
\end{equation}
and the equality of the first inequality holds only for $x=x_{i-1}, x_i$.
Then
\begin{equation*}
  \begin{split}
    \Theta_4(r)
    &=
    \sum_{i=1}^k \int_{x_{i-1}}^{x_i} \!\frac{\dd x}{\sqrt{f(x) -x^2}} \\
    &<
    \sum_{i=1}^k \int_{x_{i-1}}^{x_i} \!\frac{\dd x}{\sqrt{f(x_i)+\gamma(x_{i-1},x_i)(x-x_i)-x^2}} \\
    &=: \Lambda_k(r).
  \end{split}
\end{equation*}
Again by the formula \eqref{eq:40} of Lemma \ref{lem000}, $\Lambda_k(r)$ can be
expressed as
\begin{equation}\label{eq:47}
  \Lambda_k(r)
  =\sum_{i=1}^k \tilde{\Lambda}\left( 1+\frac{r-1}{k}(i-1), 1+\frac{r-1}{k}i \right),
\end{equation}
where
\begin{multline} \label{eq:46}
  \tilde{\Lambda}(y,z) =
  2\arcsin \left(
    \frac{1}{2} +\frac{z-\frac{1}{2}\gamma(y,z)}{\sqrt{\gamma(y,z)^2+4[f(z)-\gamma(y,z)z]}}
  \right)^{\frac{1}{2}} \\
  -2\arcsin\left(
    \frac{1}{2} +\frac{y-\frac{1}{2}\gamma(y,z)}{\sqrt{\gamma(y,z)^2+4[f(z)-\gamma(y,z)z]}}
  \right)^{\frac{1}{2}}.
\end{multline}
Therefore, the proof of this lemma is completed.
\end{proof}

\begin{lemma}\label{lem000}
If the quadratic polynomial $a+bx-x^2$ is nonnegative on the interval
$[x_j,x_i]$, then we have
\begin{equation}\label{eq:40}
  \int_{x_j}^{x_i} \!\frac{\dd x}{\sqrt{a+bx-x^2}}
  =2\arcsin\sqrt{\frac{1}{2} +\frac{x_i-\frac{1}{2}b}{\sqrt{b^2+4a}}} 
  -2\arcsin\sqrt{\frac{1}{2} +\frac{x_j-\frac{1}{2}b}{\sqrt{b^2+4a}}}.
\end{equation}
\end{lemma}

\begin{proof}{}
  Obviously, we can write
\begin{equation*}
a+bx-x^2=(x^+-x)(x-x^-),
\end{equation*}
where
\begin{equation*}
  x^{\pm}=\frac{b}{2} \pm \frac{\sqrt{b^2+4a}}{2}.
\end{equation*}
Noting that $x^-\leq x_j<x_i\leq x^+$, we have
\begin{equation*}
  \begin{split}
    \int_{x_j}^{x_i} \!\frac{\dd x}{\sqrt{a+bx-x^2}}
    &=
    \int_{x_j}^{x_i} \!\frac{\dd x}{\sqrt{(x^+-x)(x-x^-)}} \\
    &=
    2\arcsin\sqrt{\frac{x_i-x^-}{x^+-x^-}} 
    -2\arcsin\sqrt{\frac{x_j-x^-}{x^+-x^-}},
  \end{split}
\end{equation*}
which is just \eqref{eq:40}. The proof of this lemma is completed.
\end{proof}

For any given $r>1$, there are many numerical methods to evaluate the integral
$\Theta_4(r)$.
Noting that $\lambda_k(r)$ and $\Lambda_k(r)$ are elementary functions which can
be evaluated with an arbitrarily small error for each $r>1$,
Lemma \ref{lem001} provides a simple way to obtain correct lower and upper
bounds of $\Theta_4(r)$ for each $r>1$. 
However, one cannot expect to use this numerical method to obtain their correct
lower and upper bounds on any interval. 
We will take advantage of the following lemma.

\begin{lemma}\label{lem002}
$\frac{\Theta_4(r)}{\log r}$ is strictly decreasing and strictly convex for $r>1$.
\end{lemma}

\begin{proof}{}
  Recalling the expression \eqref{eq:31}, there is
\begin{equation*} 
  \frac{\Theta_4(r)}{\log r}
  =\int_0^1 \!\frac{r^t\dd t}{\sqrt{\bigl[ 1+(r^4-1)t \bigr]^{\frac{1}{2}} -r^{2t}}}.
\end{equation*}
For each given $0<t<1$, let
\begin{equation*} 
  T(r) =\frac{r^t}{\sqrt{\bigl[ 1+(r^4-1)t \bigr]^{\frac{1}{2}} -r^{2t}}}, \quad\forall r>1.
\end{equation*}
To prove this lemma, we only need to prove $T'(r)<0$ and $T''(r)>0$.

For simplicity, write $A(r)=1-t+r^4t$. Then $A'(r)=4tr^3$, and
\begin{equation*}
  T'(r)= -\frac{t (1-t) \left(r^4-1\right) r^{t-1}}{
    A^{\frac{1}{2}} \bigl( A^{\frac{1}{2}} -r^{2t} \bigr)^{\frac{3}{2}}}
  <0.
\end{equation*}
For $T''(r)$, we have
\begin{equation*}
  \frac{T''(r)}{T'(r)}
  = 
  \frac{4r^3}{r^4-1}
  +\frac{t-1}{r}
  -\frac{2tr^3}{A}
  -\frac{3(A^{-\frac{1}{2}}tr^3-tr^{2t-1})}{ A^{\frac{1}{2}} -r^{2t}},
\end{equation*}
namely,
\begin{equation*}
  \begin{split}
    (r^4-1)rA(A^{\frac{1}{2}} -r^{2t}) \frac{T''}{T'}
    &=
    A^{\frac{1}{2}}
    \left[ (t^2-2t)r^8 +(3+4t-2t^2)r^4 +1-2t+t^2 \right] \\
    &\hskip1.1em
    -r^{2t}
    \left[ (t-2t^2)r^8 +(3-2t+4t^2)r^4 +1+t-2t^2 \right].
  \end{split}
\end{equation*}
Let $R=r^4$, and
\begin{equation} \label{eq:36}
  \beta(R)=(t^2-2t)R^2 +(3+4t-2t^2)R +1-2t+t^2.
\end{equation}
Then, we have
\begin{multline} \label{eq:34}
  (r^4-1)rA(A^{\frac{1}{2}} -r^{2t}) \frac{T''(r)}{T'(r)} \\
  =
  (1-t+tR)^{\frac{1}{2}} \beta(R) 
  -\left[ \beta(R)+3t(1-t)(R-1)^2 \right] R^{\frac{t}{2}}.
\end{multline}
In order to prove $T''(r)>0$, we shall prove that the right hand side of
\eqref{eq:34} is less than zero for any $R>1$.

When $\beta(R)\leq0$, i.e., $R\geq 1+\frac{3+\sqrt{9+32 t-16 t^2}}{4 t-2 t^2}$.
Observing that
\begin{equation*}
  1-t+tR>R^t, \quad \forall R>1,
\end{equation*}
the right hand side of \eqref{eq:34} can be rewritten as
\begin{equation*}
  \left[
    (1-t+tR)^{\frac{1}{2}}
    - R^{\frac{t}{2}} 
  \right]
  \beta(R) 
  -3t(1-t)(R-1)^2 R^{\frac{t}{2}},
\end{equation*}
which is obviously less than zero.

Consider the remaining case
\begin{equation}\label{eq:35}
  1<R<1+\frac{3+\sqrt{9+32 t-16 t^2}}{4 t-2 t^2},
\end{equation}
where $\beta(R)>0$.
Let
\begin{equation*}
  \phi(R)=
  \frac{1}{2} \log (1-t+tR)
  +\log \beta(R) 
  -\log \left[ \beta(R)+3t(1-t)(R-1)^2 \right]
  -\frac{t}{2} \log R.
\end{equation*}
We want to prove that $\phi(R)<0$ for $R$ restricted within the range of \eqref{eq:35}.
Recalling \eqref{eq:36}, we have
\begin{equation*}
  \beta'(R)=(2t^2-4t)R + 3+4t-2t^2,
\end{equation*}
and then
\begin{equation*}
  \begin{split}
    \phi'(R)
    &=
    \frac{t}{2 (1-t+tR)}
    -\frac{t}{2 R}
    +\frac{\beta '(R)}{\beta (R)}
    -\frac{\beta '(R)+6t(1-t)(R-1)}{\beta (R)+3t(1-t)(R-1)^2} \\
    &=
    t(1-t)(R-1)
    \left(
      \frac{1}{2R(1-t+tR)}
      -\frac{3(3R+5)}{\beta(R)[\beta (R)+3t(1-t)(R-1)^2]}
    \right).
  \end{split}
\end{equation*}
Let
\begin{equation*}
  \varphi(R) =
  6R(1-t+tR)(3R+5)
  -\beta(R)[\beta (R)+3t(1-t)(R-1)^2],
\end{equation*}
which is a quartic polynomial of $R$.
Simple computations show that
\begin{equation} \label{eq:38}
  \varphi(1)=32,
  \quad
  \varphi'(1)=6 (7+8 t),
  \quad
  \varphi''(1)=2(9+70t+4t^2), 
\end{equation}
and that
\begin{equation*}
  \varphi'''(R)=
  6 t \left[
    (8t -20t^2 +8t^3)R
    +21 -5 t +20 t^2 -8 t^3
  \right].
\end{equation*}
We can prove that $\varphi'''(R)>0$ when $R$ is within the range of \eqref{eq:35}.
In fact, note that
\begin{equation*}
  \varphi'''\left(
    1+\frac{3+\sqrt{9+32 t-16 t^2}}{4 t-2 t^2}
  \right)
  = 6t
  \left(
    27-9t
    +(2-4t)\sqrt{9+32 t-16 t^2}
  \right).
\end{equation*}
Since
\begin{equation*}
  \frac{\dd}{\dd t} 
  \left(
    27-9t
    +(2-4t)\sqrt{9+32 t-16 t^2}
  \right)
  =
  -9
  -\frac{4(1+56t-32t^2)}{\sqrt{9+32 t-16 t^2}}
  <0,
\end{equation*}
then
\begin{equation*}
  27-9t +(2-4t)\sqrt{9+32 t-16 t^2} >8,
\end{equation*}
which implies
\begin{equation*}
  \varphi'''\left(
    1+\frac{3+\sqrt{9+32 t-16 t^2}}{4 t-2 t^2}
  \right)>0.
\end{equation*}
Besides, 
\begin{equation*}
  \varphi'''(1)=18t(7+t)>0.
\end{equation*}
Due to the fact that $\varphi'''$ is a linear function of $R$,
$\varphi'''(R)>0$ for any $R$ given by \eqref{eq:35}.
Now recalling \eqref{eq:38}, we see that $\varphi(R)>0$, implying that
$\phi'(R)<0$. By $\phi(1)=0$, there is $\phi(R)<0$.
By our construction of $\phi$, the right hand side of \eqref{eq:34} is less than
zero when $R$ is within the range of \eqref{eq:35}.

Now we have proved that the right hand side of \eqref{eq:34} is less than zero
for every $R>1$.
Namely, $\frac{T''(r)}{T'(r)}<0$ for every $r>1$.
Thus, $T''(r)>0$ for every $r>1$.
The proof of this lemma is completed.
\end{proof}

Using Lemma \ref{lem002}, we can develop a method to verify
$\Theta_4(r)>\frac{\pi}{2}$ on one interval through evaluating the values of
$\Theta_4(r)$ at only three points.

\begin{lemma}\label{lem012}
Let $\LB$ be a lower bound of $\lambda_k$, $\UB$ an upper bound of
$\Lambda_k$, and $\Delta$ a small positive number.
For $1<r_L<\hat{r}<r_R$, if 
\begin{equation} \label{eq:51} \tag{A} 
  \LB(\hat{r})>\frac{\pi}{2},
  \quad
  \frac{\UB(\hat{r}-\Delta)}{\log (\hat{r}-\Delta)}
  >
  \frac{\LB(\hat{r})}{\log \hat{r}}
  >
  \frac{\UB(\hat{r}+\Delta)}{\log (\hat{r}+\Delta)},
\end{equation}
\begin{equation} \label{eq:48} \tag{B}
  \frac{\LB(\hat{r})}{\log \hat{r}}\log r_L
  +\left(
    \frac{\LB(\hat{r})}{\log \hat{r}}
    -\frac{\UB(\hat{r}+\Delta)}{\log (\hat{r}+\Delta)}
  \right)
  \frac{\hat{r}-r_L}{\Delta} \log r_L
  >\frac{\pi}{2},
\end{equation}
and
\begin{equation} \label{eq:49} \tag{C}
  \frac{\LB(\hat{r})}{\log \hat{r}} \log r_R
  -\left(
    \frac{\UB(\hat{r}-\Delta)}{\log (\hat{r}-\Delta)}
    -\frac{\LB(\hat{r})}{\log \hat{r}}
  \right)
  \frac{r_R-\hat{r}}{\Delta} \log r_R
  >\frac{\pi}{2},
\end{equation}
then there is
\begin{equation}\label{eq:50}
  \Theta_4(r)>\frac{\pi}{2}, \quad \forall r\in [r_L, r_R].
\end{equation}
\end{lemma}

\begin{proof}
  We first consider the case when $r\in[r_L,\hat{r}]$.
By Lemma \ref{lem002}, $\frac{\Theta_4(r)}{\log r}$ is strictly convex for
$r>1$, implying that
\begin{equation*}
  \begin{split}
    \frac{\Theta_4(r)}{\log r}    
    &\geq 
    \frac{\Theta_4(\hat{r})}{\log \hat{r}}    
    +\left(
      \frac{\Theta_4(\hat{r})}{\log \hat{r}}
      -\frac{\Theta_4(\hat{r}+\Delta)}{\log (\hat{r}+\Delta)}
    \right)
    \frac{\hat{r}-r}{\Delta} \\
    &\geq
    \frac{\LB(\hat{r})}{\log \hat{r}}
    +\left(
      \frac{\LB(\hat{r})}{\log \hat{r}}
      -\frac{\UB(\hat{r}+\Delta)}{\log (\hat{r}+\Delta)}
    \right)
    \frac{\hat{r}-r}{\Delta}.
  \end{split}
\end{equation*}
Let
\begin{equation*}
  \phi(r)=
  \frac{\LB(\hat{r})}{\log \hat{r}} \log r
  +\left(
    \frac{\LB(\hat{r})}{\log \hat{r}}
    -\frac{\UB(\hat{r}+\Delta)}{\log (\hat{r}+\Delta)}
  \right)
  \frac{\hat{r}-r}{\Delta} \log r,
\end{equation*}
then $\Theta_4(r)\geq \phi(r)$ for $r\in[r_L,\hat{r}]$.
Now by the condition \eqref{eq:51}, one can easily see that $\phi(r)$ is
strictly concave, and $\phi(\hat{r})=\LB(\hat{r})>\frac{\pi}{2}$.
Besides, the condition \eqref{eq:48} says $\phi(r_L)>\frac{\pi}{2}$.
Therefore,
\begin{equation*}
  \phi(r)>\frac{\pi}{2} \text{ on } [r_L,\hat{r}],
\end{equation*}
which implies \eqref{eq:50} for $r\in[r_L,\hat{r}]$.

The proof for $r\in[\hat{r},r_R]$ is similar.
In fact, by the convexity of $\frac{\Theta_4(r)}{\log r}$, there is
\begin{equation*}
  \begin{split}
    \frac{\Theta_4(r)}{\log r}    
    &\geq 
    \frac{\Theta_4(\hat{r})}{\log \hat{r}}    
    -\left(
      \frac{\Theta_4(\hat{r}-\Delta)}{\log (\hat{r}-\Delta)}
      -\frac{\Theta_4(\hat{r})}{\log \hat{r}}
    \right)
    \frac{r-\hat{r}}{\Delta} \\
    &\geq 
    \frac{\LB(\hat{r})}{\log \hat{r}}
    -\left(
      \frac{\UB(\hat{r}-\Delta)}{\log (\hat{r}-\Delta)}
      -\frac{\LB(\hat{r})}{\log \hat{r}}
    \right)
    \frac{r-\hat{r}}{\Delta}.
  \end{split}
\end{equation*}
Now by conditions \eqref{eq:51} and \eqref{eq:49}, one will see that
\eqref{eq:50} is true for any $r\in[\hat{r},r_R]$.
\end{proof}

With Lemmas \ref{lem001} and \ref{lem012}, we can prove
$\Theta_4(r)>\frac{\pi}{2}$ for $r\in(1.22,8.62)$ by the numerical estimation.

\begin{lemma} \label{lem013}
$\Theta_4(r)>\frac{\pi}{2}$ holds for any $r\in(1.22,8.62)$.
\end{lemma}

\begin{proof}
We let $\Delta=\frac{1}{1000}$ in this proof.
For $\hat{r}=1.229$, we compute $\lambda_k, \Lambda_k \,(k=20000)$ with an error
of less than $10^{-9}$, and then obtain the following correct bounds:
\begin{equation*}
  \UB(\hat{r}-\Delta)=1.573533,
  \quad
  \LB(\hat{r})=1.573553,
  \quad
  \UB(\hat{r}+\Delta)=1.573577,
\end{equation*}
satisfying the condition \eqref{eq:51}.
Let $r_L=1.220$, then the left hand side of the condition \eqref{eq:48} minus
its right hand side is approximately equal to $1.2\times10^{-4}$.
Let $r_R=1.238$, then the left hand side of the condition \eqref{eq:49} minus
its right hand side is approximately equal to $5.1\times10^{-4}$.
By Lemma \ref{lem012}, $\Theta_4(r)>\frac{\pi}{2}$ holds on $[1.22,1.238]$.
See the first row of the following table.
The computations for other rows are similar, except that $k=80000$ for the last
row. 

\renewcommand{\arraystretch}{1.2}
\rowcolors{2}{gray!15}{gray!30}
\begin{longtable}{cccccccc}
  $\hat{r}$
  & $\UB(\hat{r}-\Delta)$
  & $\LB(\hat{r})$
  & $\UB(\hat{r}+\Delta)$
  & $r_L$
  & $r_R$
  & \eqref{eq:48}$\cdot10^4$
  & \eqref{eq:49}$\cdot10^4$ \\
  \endfirsthead
  \rowcolor{white}
  $\hat{r}$
  & $\UB(\hat{r}-\Delta)$
  & $\LB(\hat{r})$
  & $\UB(\hat{r}+\Delta)$
  & $r_L$
  & $r_R$
  & \eqref{eq:48}$\cdot10^4$
  & \eqref{eq:49}$\cdot10^4$ \\
  \endhead
  \rowcolor{white} \multicolumn{8}{r}{\textit{Continued on next page}} \\
  \endfoot
  \endlastfoot
  1.229
  & 1.573533
  & 1.573553
  & 1.573577
  & 1.220
  & 1.238           
  & 1.2
  & 5.1 \\
  1.248
  & 1.573953
  & 1.573975
  & 1.573999
  & 1.238
  & 1.259
  & 4.3
  & 4.1 \\
  1.271
  & 1.574491
  & 1.574514
  & 1.574539
  & 1.259
  & 1.285
  & 4.7
  & 0.77 \\
  1.300
  & 1.575211
  & 1.575235
  & 1.575262
  & 1.285
  & 1.317
  & 3.7
  & 2.0 \\
  1.336
  & 1.576165
  & 1.576190
  & 1.576219
  & 1.317
  & 1.357
  & 2.7
  & 4.1 \\
  1.381
  & 1.577438
  & 1.577466
  & 1.577497
  & 1.357
  & 1.409
  & 4.0
  & 0.78 \\
  1.441
  & 1.579256
  & 1.579286
  & 1.579319
  & 1.409
  & 1.478
  & 3.2
  & 3.1 \\
  1.522
  & 1.581882
  & 1.581914
  & 1.581949
  & 1.478
  & 1.574
  & 2.9
  & 2.1 \\
  1.638
  & 1.585887
  & 1.585921
  & 1.585958
  & 1.574
  & 1.715
  & 1.1
  & 2.3 \\
  1.814
  & 1.592262
  & 1.592297
  & 1.592335
  & 1.715
  & 1.937
  & 0.54
  & 1.9 \\
  2.104
  & 1.602875
  & 1.602909
  & 1.602947
  & 1.937
  & 2.321
  & 1.6
  & 1.9 \\
  2.641
  & 1.621023
  & 1.621050
  & 1.621085
  & 2.321
  & 3.075
  & 2.0
  & 3.2 \\
  3.770
  & 1.649591
  & 1.649599
  & 1.649631
  & 3.075
  & 4.720
  & 5.9
  & 6.5 \\
  6.500
  & 1.682683
  & 1.682680
  & 1.682697
  & 4.720
  & 8.870          
  & 32
  & 7.8
\end{longtable}

The conclusion of this lemma follows from these numerical computations in the
above table.
\end{proof}

We remark that $\lambda_k$ and $\Lambda_k$ are not efficient approaches to
estimate lower and upper bounds of $\Theta_4$.
If an efficient approximation method is used to compute correct lower and upper bounds
of $\Theta_4(r)$ for each given $r$, the number of rows in the above table will
be reduced.
Here we use $\lambda_k$ and $\Lambda_k$ because they are very simple and
computable. The computation only involves computing the values of elementary
functions at some points, so the strictly correct lower and upper bounds of these
values can be obtained.

Combining Lemmas \ref{lem010}, \ref{lem011} and \ref{lem013}, we have proved the
inequality \eqref{eq:69}, which together with the inequality \eqref{eq:23} and
Lemma \ref{lem01} implies \eqref{eq:10}. 
Thus, Lemma \ref{lem008} holds for the case when $1\leq q\leq4$.

Note that Lemma \ref{lem008} for $q>4$ follows directly from
Lemmas \ref{lem009} and \ref{lem010}.

\section*{Conflict of interest}
On behalf of all authors, the corresponding author states that there is no conflict of interest.


\end{document}